\documentclass{article}
\usepackage{arxiv}
\usepackage{graphicx,amsthm,amsmath,amsfonts,amssymb,txfonts,cmll}
\usepackage{natbib,cancel,url,textcomp}
\usepackage{hyperref}
\usepackage{cleveref}
\title{Compactness in Constructive Mathematics via Affine Logic}
\author{Kazumi Kasaura}
\date{}

\theoremstyle{definition}
\newtheorem{theorem}{Theorem}
\newtheorem{proposition}[theorem]{Proposition}
\newtheorem{lemma}[theorem]{Lemma}
\newtheorem{corollary}[theorem]{Corollary}
\newtheorem{definition}[theorem]{Definition}

\newtheorem{remark}[theorem]{Remark}

\theoremstyle{remark}

\DeclareFontFamily{U}{MnSymbolC}{}
\DeclareSymbolFont{MnSyC}{U}{MnSymbolC}{m}{n}
\DeclareMathSymbol{\diamondplus}{\mathbin}{MnSyC}{"7C}
\DeclareMathSymbol{\diamonddot}{\mathbin}{MnSyC}{"7E}
\DeclareFontShape{U}{MnSymbolC}{m}{n}{
    <-6>  MnSymbolC5
   <6-7>  MnSymbolC6
   <7-8>  MnSymbolC7
   <8-9>  MnSymbolC8
   <9-10> MnSymbolC9
  <10-12> MnSymbolC10
  <12->   MnSymbolC12}{}

\DeclareMathOperator{\inte}{\mathsf{int}}
\DeclareMathOperator{\clo}{\mathsf{cl}}
\DeclareMathOperator{\ofcourse}{\mathop{!}}
\DeclareMathOperator{\whynot}{\mathop{?}}
\def\subsetoc#1{#1^{!}}
\def\subsetwn#1{#1^{?}}

\def\Prop{\mathbf{\Omega}}
\def\P#1{\mathcal{P}(#1)}
\def\O{\mathcal{O}}
\def\linimp{\multimap}

\def\mand{\otimes}

\def\aand{\with}
\def\mor{\parr}
\def\aor{\oplus}
\def\mcap{\mathbin{\boxtimes}}
\def\acap{\mathbin{\sqcap}}
\def\mcup{\mathbin{\diamondplus}}
\def\acup{\mathbin{\sqcup}}
\def\atrue{\top}

\def\mfalse{\bot}
\def\not#1{{#1}^{\bot}}
\def\compl#1{{#1}^{\mathrm{c}}}
\def\N{\mathbb{N}}
\def\Q{\mathbb{Q}}

\def\Set#1#2{\left\{#1\ \middle|\ #2\right\}}

\def\atimes{\mathbin{\times^{\aand}}}
\def\catimes{\mathbin{\boldsymbol{+}^{\aor}}}
\def\mtimes{\mathbin{\times^{\mand}}}

\makeatletter
\newcommand\bigmcup{\mathop{\mathpalette\big@mcup\relax}}
\newcommand\big@mcup[2]{%
  \vcenter{\hbox{\m@th
    \scalebox{\ifx#1\displaystyle 2\else1.2\fi}{$#1\mcup$}%
  }}%
}
\newcommand\bigmor{\mathop{\mathpalette\big@mor\relax}}
\newcommand\big@mor[2]{%
  \vcenter{\hbox{\m@th
    \scalebox{\ifx#1\displaystyle 2\else1.2\fi}{$#1\mor$}%
  }}%
}
\newcommand\bigmand{\mathop{\mathpalette\big@mand\relax}}
\newcommand\big@mand[2]{%
  \vcenter{\hbox{\m@th
    \scalebox{\ifx#1\displaystyle 2\else1.2\fi}{$#1\mand$}%
  }}%
}
\newcommand\bigmcap{\mathop{\mathpalette\big@mcap\relax}}
\newcommand\big@mcap[2]{%
  \vcenter{\hbox{\m@th
    \scalebox{\ifx#1\displaystyle 2\else1.2\fi}{$#1\mcap$}%
  }}%
}

\makeatother

\begin{document}

\maketitle

\begin{abstract}
We study topology, particularly compactness, as an extension of Shulman's work on constructive mathematics via affine logic, while allowing propositional impredicativity.
We introduce a notion of compactness in affine logic and prove the fundamental properties of compactness, including the extreme value theorem and the Heine-Borel theorem for 'cuts', which are a version of Dedekind cuts in affine logic.
Moreover, from the antithesis translation of the Heine-Borel theorem for cuts to intuitionistic logic, we derive the Heine-Borel theorem for one-sided reals intuitionistically, and have verified the proof with an interactive theorem prover.
The code is available at \url{https://github.com/hziwara/CutsHeineBorel}.
\end{abstract}

\section{Introduction}

\cite{shulman2022affine} introduces constructive mathematics via affine logic and a method called 'antithesis translation' to convert a proposition in affine logic to one in intuitionistic logic.
Moreover, it is demonstrated that notions in existing constructive mathematics can be derived through translation from affine logic.
Specifically, in topology, the notion of a topology in affine logic is introduced, and it is shown that its translation includes existing constructive notions for topologies as special cases.

This paper extends this line of work in topology, focusing particularly on compactness.
While the equality structure (Bishop sets) is considered and subsets are defined as predicates compatible with this structure in \cite{shulman2022affine}, we do not adopt this framework.
The reason is its incompatibility with multiplicative logical operators.
See Remark~\ref{remark:equation}.
Instead, we treat all functions to the type of propositions as subsets.

First, in Section~\ref{sec:topology}, we discuss some basic notions of topology.
We introduce an affine version of the notion of compactness and prove some properties corresponding to classical results in Section~\ref{sec:compactness}.
Also, in Section~\ref{sec:cuts}, we introduce the topology on the type of 'cuts' defined in \cite{shulman2022affine}, which is an affine version of the extended Dedekind reals and corresponds to unbounded intervals in the antithesis translation.
We prove the extreme value theorem and the Heine-Borel theorem for the topology.
The proof of the latter depends on the assumption of propositional impredicativity.

In addition, in Section~\ref{sec:translation}, we consider the antithesis translation of the Heine-Borel theorem for cuts.
From this, we derive the proof of the Heine-Borel theorem for one-sided reals and have verified it with the interactive theorem prover.
While this proof is intuitionistic, it shows only existence without giving a way to construct it, because it relies essentially on propositional impredicativity. Finally, we discuss the relationship between our results and the previous work in Section~\ref{sec:related work}.

\section{Notation}

Unlike \cite{shulman2022affine}, we use Girard's original symbols $\mand$, $\mor$, $\aand$, $\aor$ for logical operators, while we use Shulman's symbols $\mcap$, $\mcup$, $\acap$, $\acup$ for operations on subsets.
This may cause some confusion (especially, $\aor$ does not correspond to $\mcup$), but it has the advantage of allowing us to distinguish between symbols for logical operations and set operations.
Also, we use the standard symbols $\forall$, $\exists$, $\subseteq$, etc., even for affine logic. This is because we do not consider intuitionistic logic except \S~\ref{sec:translation}, so there is no need to distinguish between affine ones and intuitionistic ones.
We assign higher precedence to $\mand, \aand$ over $\mor, \aor, \linimp$ in logical expressions.

Let $\Prop$ be the type of propositions.
We assume propositional impredicativity in the sense that propositions referring to $\Prop$ itself are also in $\Prop$.
For a type $X$, let $\P{X}:=(X\to \Prop)$ be the type of subsets on $X$.
We often write $x\in s$ (resp. $x \notin s$) instead of $s(x)$ (resp. $\not{s(x)}$) for $x:X$ and $s:\P{X}$ and define $s$ in the form of $\Set{x:X}{s(x)}$.

For $s,t : \P{X}$ and an indexed family $u_{\_} : \iota \to \P{X}$,
we define them as follows:
\begin{align*}
s\subseteq t &:= \forall x:X, x\in s \linimp x \in t,
&\compl{s} &:= \Set{x:X}{x\notin s} \\
s \mcap t &:= \Set{x:X}{x\in s \mand x \in t},
&s \mcup t &:= \Set{x:X}{x\in s \mor x \in t}, \\
s \acap t &:= \Set{x:X}{x\in s \aand x \in t},
&s \acup t &:= \Set{x:X}{x\in s \aor x \in t}, \\
\subsetoc{s} &:= \Set{x:X}{\ofcourse{(x\in s)}},
&\subsetwn{s} &:= \Set{x:X}{\whynot{(x\in s)}}, \\
\bigcap_{i:\iota} u_i &:= \Set{x:X}{\forall i: \iota,\, x \in u_i},
&\bigcup_{i:\iota} u_i &:= \Set{x:X}{\exists i: \iota,\, x \in u_i}.
\end{align*}
Abusing the symbol, the whole set $\Set{x:X}{\atrue}$ is also denoted by $X$.
The empty set $\emptyset : \P{X}$ is defined as $\Set{x:X}{\mfalse}$.

We say that a subset $s:\P{X}$ is \textit{decidable} if $x\in s \aor x\in \compl{s}$ for any $x:X$.

For $n:\N$, we denote the type of the natural numbers less than $n$ by $[n]$.
While the standard quantifiers $\forall$ and $\exists$ are additive, we can inductively define $\bigmand_{i:[n]} P_i$ and $\bigmor_{i:[n]} P_i$ for a finitely indexed family $P_{\_}:[n] \to \Prop$ of propositions.
We also define $\bigmcap_{i:[n]} u_i$ and $\bigmcup_{i:[n]} u_i$ for a finitely indexed family $u_{\_}:[n] \to \P{X}$ of subsets.

\begin{remark}\label{remark:equation}
In \cite{shulman2022affine}, a set $X$ is defined as a type equipped with an equality relation,  and a subset is defined as a function $s: X\to \Prop$ such that $x=y \mand x\in s \linimp y \in s$.
Since the multiplicative intersection $s\mcap t$ of two subsets $s,t$ is not always a subset, the operator is defined as
\[
s\hat{\mcap} t :=\Set{x:X}{\exists y:X, \left(x=y \mand y \in s \mand y \in t\right)}.
\]
However, this operator is not even associative:
\begin{align*}
\left(s\hat{\mcap} t \right)\hat{\mcap} u&=\Set{x:X}{\exists y,z:X, \left(x=y \mand y=z\mand z \in s \mand z \in t\mand y \in u\right)}, \\
s\hat{\mcap} \left(t \hat{\mcap} u\right)&=\Set{x:X}{\exists y,z:X, \left(x=y \mand y=z \mand y\in s \mand z \in t\mand z \in u\right)}.
\end{align*}
This makes it difficult to consider the multiplicative intersection of finitely many subsets, which is important in our study.
Thus, we do not consider equality relations, and any function to $\Prop$ is treated as a subset.
\end{remark}

\section{Topology}\label{sec:topology}

In this section, after reviewing the definition of topology in affine logic, we introduce notions of open and closed sets, basis, and product topology.

First, following \cite{shulman2022affine}, we review the notion of topology by the interior operator.
In addition, we consider a Moore operator~(\cite{schechter1996handbook}) and a \v{C}ech operator~(\cite{vcech1966topological}), both of which are weaker concepts than topology.
Since these conditions are axioms, we can use them arbitrarily many times even in affine logic.

\begin{definition}\label{def:operator}
Let $\inte : \P{X} \to \P{X}$ be an operator on subsets.
It is called a \textit{Moore interior operator} if the following conditions are satisfied:
\begin{enumerate}
\item[\textbf{I1}.] $\forall s: \P{X}, \inte s \subseteq s$,
\item[\textbf{I2}.] $\forall s,t: \P{X}, \left(s \subseteq t \linimp \inte s \subseteq \inte t\right)$,
\item[\textbf{I3}.] $\forall s: \P{X}, \inte s \subseteq \inte (\inte s)$.\label{enum:transitivity}
\end{enumerate}
An \textit{interior operator} is a Moore interior operator satisfying additionally the following conditions:
\begin{enumerate}
\item[\textbf{I4}.] $X \subseteq \inte X$,\label{enum:operator_univ}
\item[\textbf{I5}.] $\forall s,t: \P{X}, \inte s \mcap \inte t \subseteq \inte \left(s \acap t\right)$.\label{enum:operator_interaction}
\end{enumerate}
An operator is called a \textit{\v{C}ech interior operator} if it satisfies the conditions for interior operator except \textbf{I3}.
\end{definition}

As a classical definition, we introduce the definition of topology as a set of open sets.

\begin{definition}
A \textit{Moore collection of open subsets} is a subset $\O : \P{\P{X}}$ of the power set satisfying the following conditions:
\begin{enumerate}
\item[\textbf{O1}.] $\forall s,t: \P{X},\left(s\in \O \linimp s\subseteq t\mand t\subseteq s \linimp t \in \O\right)$.
\item[\textbf{O2}.] $\forall s: \P{X}, \Set{x:X}{\exists t :\P{X}, \left(t\in \O \mand t\subseteq s \mand x \in t \right)}\in\O$.
\end{enumerate}
A \textit{collection of open subsets} is a Moore collection of open subsets satisfying additionally the following conditions:
\begin{enumerate}
\item[\textbf{O3}.] $X \in \O$,
\item[\textbf{O4}.] $\forall s,t: \P{X}, \left(s \in \O \mand t \in \O \linimp \exists u : \P{X}, \left(u \in \O \mand s \mcap t \subseteq u \mand u \subseteq s \acap t\right)\right)$,
\end{enumerate}
\end{definition}

\textbf{O1} says that $\O$ is a 'subset' in the sense of \cite{shulman2022affine}.
\textbf{O2} corresponds to the condition that any union of open subsets is open. See Remark~\ref{remark:union_of_open_subsets}.
\textbf{O3} is the same as in the classical definition and \textbf{O4} corresponds to the condition that any intersection of two open sets is open.

\begin{proposition}\label{prop:topology_def}
Moore interior operators correspond one-to-one to Moore collections of open subsets, and interior operators correspond one-to-one to collections of open subsets.
\end{proposition}
\begin{proof}
From a Moore interior operator $\inte$, we can define $\O$ as
\begin{equation}\label{eq:open_def}
\O := \{s\mid s \subseteq \inte s\}.    
\end{equation}
Then, \textbf{O1} follows from \textbf{I2}.
We show \textbf{O2}. For any $s:\P{X}$, let \[s':=\Set{x:X}{\exists t :\P{X}, \left(t\in \O \mand t\subseteq s \mand x \in t \right)}.\]
Because \textbf{I3} means $\inte s \in \O$, we have $\inte s \subseteq s'$ by taking $\inte s$ as $t$.
Moreover, $s' \subseteq \inte s$ follows from \textbf{I2}. Thus, using $\inte s \in \O$ again, $s' \in \O$ is derived from \textbf{O1}.


Conversely, from a Moore collection $\O$ of open sets, we can define the interior operator as \begin{equation}\label{eq:interior_def}
\inte s := \Set{x:X}{\exists t :\P{X}, \left(t\in \O \mand t\subseteq s \mand x \in t \right)}.
\end{equation}
Then, \textbf{I1} and \textbf{I2} are trivial.
We can show that 
$s \in \O \linimp s\subseteq \inte s$ by taking $s$ itself as $t$. Since \textbf{O2} means $\inte s \in \O$, \textbf{I3} holds true.
Moreover, since $\inte s \subseteq s$ and $\inte s \in \O$, $s \subseteq \inte s \linimp s \in \O$ follows from \textbf{O1}. Thus, the constructions (\ref{eq:open_def}) and (\ref{eq:interior_def}) are inverse to each other.

We prove that interior operators correspond to collections of open sets via this correspondence.
Clearly, \textbf{O3} is equivalent to \textbf{I4}. \textbf{O4} can be proven from \textbf{I5} by taking $u:=\inte \left(s\acap t\right)$.
We show \textbf{I5} from \textbf{O4}. Since $\inte s \in \O$ and $\inte t \in \O$, we can take $u:\P{X}$ such that \[u \in \O\mand \inte s \mcap \inte t \subseteq u \mand u \subseteq \inte s \acap \inte t.\]
Since $\inte s \acap \inte t \subseteq s \acap t$, we have $u \subseteq \inte \left(s \acap t\right)$ from the first and third conditions. Thus, \textbf{I5} holds.

\end{proof}

We define a \textit{topology} on $X$ as a structure defined by an interior operator, or equivalently, by a collection of open subsets.

For an operator $\inte: \P{X}\to \P{X}$, we define the dual operator $\clo: \P{X}\to \P{X}$ as $\clo s := \compl{(\inte \compl{s})}$.
If $\inte$ is a (Moore / \v{C}ech) interior operator, $\clo$ is called \textit{(Moore / \v{C}ech) closure operator}.
Moreover, if $\inte$ is a (Moore) interior operator, the \textit{(Moore) collection of closed subsets} is defined as $\mathcal{C} := \{s : \P{X} \mid \compl{s} \in \O\}$.\footnote{Indeed, the term 'Moore collection' usually refers to Moore collections of closed subsets, and Moore closures are considered instead of interiors. However, we follow the convention that topology is often defined using open subsets, and that \cite{shulman2022affine} defines topology using interior.}
An element in $\mathcal{O}$ is called \textit{open} and that in $\mathcal{C}$ is called \textit{closed}.
Axioms for $\clo$ or $\mathcal{C}$ can be written down, and
these notions also determine the topology.


\begin{remark}\label{remark:union_of_open_subsets}
Let $X$ be a type equipped with a Moore operator.
For an indexed family of subsets $\mathcal{S} : A \to \P{X}$, if all members are open, their union is also open, because
\[\bigcup_{\alpha:A} \mathcal{S}(\alpha) \subseteq\bigcup_{\alpha:A} \inte \mathcal{S}(\alpha)\subseteq \inte \bigcup_{\alpha:A} \mathcal{S}(\alpha).\]
However, for a subset of the power set $\mathcal{S}: \P{\P{X}}$, it is required that $\mathcal{S}\subseteq \mathcal{S}\mcap \O$, which holds when $\mathcal{S}\subseteq\O$ and $\mathcal{S}$ is affirmative ($\mathcal{S}\subseteq \subsetoc{\mathcal{S}}$), to prove \[\Set{x:X}{\exists s : \P{X}, s\in \mathcal{S} \mand x \in s}\in \O,\]
because both $s\in \mathcal{S}$ and $s\in \mathcal{O}$ are needed to prove
\[
s\subseteq \inte\Set{x:X}{\exists s : \P{X}, s\in \mathcal{S} \mand x \in s}.
\]
For this reason, we prefer indexed families of subsets to subsets of subsets.
\end{remark}

Next, we treat the determination of topology by basis.
\begin{definition}
An indexed family $\mathcal{B}:\mathcal{I}\to \P{X}$ of subsets is called a \textit{basis} if
\[\forall x:X, \exists \iota:\mathcal{I}, x\in\mathcal{B}(\iota)\]
and
\begin{align*}
\forall x:X,\iota_0,\iota_1:\mathcal{I}, (&x\in\mathcal{B}(\iota_0)\mand x\in\mathcal{B}(\iota_1)\linimp \\
&\exists \iota_2:\mathcal{I}, \left(x\in\mathcal{B}(\iota_2)\mand\left(\mathcal{B}(\iota_2)\subseteq\mathcal{B}(\iota_0) \aand \mathcal{B}(\iota_2)\subseteq\mathcal{B}(\iota_1)\right)\right)).
\end{align*}
\end{definition}
\begin{proposition}
For an indexed family $\mathcal{B}:\mathcal{I}\to \P{X}$ of subsets, we define an operator $\inte:\P{X}\to\P{X}$ as
\[
\inte s := \Set{x:X}{\exists \iota:\mathcal{I}, x\in \mathcal{B}(\iota)\mand \mathcal{B}(\iota)\subseteq s}.
\]
Then, $\inte$ is a Moore interior operator.
Moreover, if $\mathcal{B}$ is a basis, $\inte$ is an interior operator.
\end{proposition}
\begin{proof}
The conditions \textbf{I1} and \textbf{I2} are clear.

We show \textbf{I3}. We assume $x \in \inte s$. We can take $\iota: \mathcal{I}$ such that $x \in \mathcal{B}(\iota) \mand \mathcal{B}(\iota) \subseteq s$. Then, since
\[\mathcal{B}(\iota) \subseteq s \linimp y \in \mathcal{B}(\iota) \linimp y \in \inte s,\]
$\mathcal{B}(\iota) \subseteq \inte s$. Thus, $x \in \inte \left(\inte s\right)$.

The condition \textbf{I4} follows immediately from the first condition for basis.

We show \textbf{I5} assuming the second condition for basis. We assume $x \in \inte s \mand x \in \inte t$.
We can take $\iota_0,\iota_1: \mathcal{I}$ such that $x \in \mathcal{B}(\iota_0) \mand \mathcal{B}(\iota_0) \subseteq s$ and $x \in \mathcal{B}(\iota_1) \mand \mathcal{B}(\iota_1) \subseteq t$.
From $x \in \mathcal{B}(\iota_0) \mand x \in \mathcal{B}(\iota_1)$, we can take $\iota_2 : \mathcal{I}$ as in the condition. Because
\[\mathcal{B}(\iota_0) \subseteq s \mand \mathcal{B}(\iota_1) \subseteq t \mand \left(\mathcal{B}(\iota_2)\subseteq\mathcal{B}(\iota_0) \aand \mathcal{B}(\iota_2)\subseteq\mathcal{B}(\iota_1)\right)\linimp \mathcal{B}(\iota_2) \subseteq s \acap t,\]
we have $x \in \inte (s \acap t)$.
\end{proof}


\begin{remark}
    For a similar reason to Example 10.3 in \cite{shulman2022affine}, it is necessary for the above proof that the conjunction in the left-hand side of \textbf{I5} is multiplicative and that in the right-hand side is additive.
\end{remark}

We consider the product of two spaces.
For two spaces $X,Y$ and subsets $s_X : \P{X}$, $s_y:\P{Y}$, we define the product subsets in two ways as
\[s_X\atimes s_Y := \Set{(x,y):X\times Y}{x \in s_X \aand y \in s_Y},\]
\[s_X\mtimes s_Y := \Set{(x,y):X\times Y}{x \in s_X \mand y \in s_Y}.\]

\def\minte{\inte^{\mand}}
\def\ainte{\inte^{\aand}}

\begin{proposition}
For two spaces $X$ and $Y$ with operators, let $\minte : \P{X\times Y} \to \P{X \times Y}$ as
\[
\minte s := \Set{z:X\times Y}{\exists u:\P{X}, v:\P{Y}, \left(z \in \inte u \mtimes \inte v \mand u \mtimes v \subseteq s\right)}.
\]
Then,
\begin{itemize}
    \item If both $X$ and $Y$ satisfy \textbf{I1}, $\minte$ also satisfies it.
    \item $\minte$ always satisfies \textbf{I2}.
    \item If both $X$ and $Y$ satisfy \textbf{I3}, $\minte$ also satisfies it.
    \item If both $X$ and $Y$ satisfy \textbf{I4}, $\minte$ also satisfies it.
    \item If both $X$ and $Y$ satisfy \textbf{I5}, $\minte$ also satisfies it.
\end{itemize}
Therefore, if $X$ and $Y$ are topological spaces, $X\times Y$ with $\minte$ is also a topological space.
\end{proposition}
\begin{proof}
The conditions \textbf{I1}, \textbf{I2}, and \textbf{I4} are immediate.

We show \textbf{I3}.
Let $z \in \minte s$.
We can take $u:\P{X},\,v:\P{Y}$ such that $z \in\inte u \mtimes \inte v$ and $u \mtimes v \subseteq s$.
Then, from $u \mtimes v \subseteq s$ and the definition, $\inte u \mtimes \inte v \subseteq \minte s$. Since $z \in \inte(\inte u) \mtimes \inte(\inte v)$ by the assumption, we have $z \in \minte(\minte s)$.

We show \textbf{I5}.
We assume $z\in \minte s \mcap \minte t$. We can take $u,u':\P{X}$ and $v,v':\P{Y}$ such that
\[
z\in \left(\inte u \mtimes \inte v\right)\mcap \left(\inte u'\mtimes \inte v'\right)\mand u\mtimes v \subseteq s \mand u'\mtimes v' \subseteq t.
\]
By the assumption,
\begin{align*}
\left(\inte u \mtimes \inte v\right)\mcap \left(\inte u'\mtimes \inte v'\right)
&\equiv (\inte u \mcap \inte u') \mtimes (\inte v \mcap \inte v')\\ &\subseteq \inte (u \acap u') \mtimes \inte (v \acap v'). 
\end{align*}
Thus, it is enough to show
\[
u \mtimes v \subseteq s \mand u' \mtimes v' \subseteq t \linimp
(u \acap u') \mtimes (v \acap v') \subseteq s \acap t.
\]
This is true in affine logic.
\end{proof}


While this definition is natural, 
we could not prove the proposition that the product of any two compact subsets is compact (Proposition~\ref{prop:compactness_product}) for this topology.
Thus, we introduce another definition.

\begin{proposition}
For two spaces $X$ and $Y$ with operators, let $\ainte : \P{X\times Y} \to \P{X \times Y}$ as
\[
\ainte s := \Set{z:X\times Y}{\exists u:\P{X},\, v:\P{Y},\, z \in \inte u \mtimes \inte v \mand u \atimes v \subseteq s}.
\]
Then,
\begin{itemize}
    \item If both $X$ and $Y$ satisfy \textbf{I1}, $\ainte$ also satisfies it.
    \item $\ainte$ always satisfies \textbf{I2}.
    \item If both $X$ and $Y$ satisfy \textbf{I4}, $\ainte$ also satisfies it.
    \item If both $X$ and $Y$ satisfy \textbf{I5}, $\ainte$ also satisfies it.
\end{itemize}
Therefore, if the operators for $X$ and $Y$ are \v{C}ech operators, $\ainte$ is also a \v{C}ech operator.
\end{proposition}
\begin{proof}
The conditions \textbf{I2} and \textbf{I4} are immediate.

\textbf{I1} follows from $u\mtimes v \subseteq u \atimes v$.

We show \textbf{I5}. By the same argument as the previous proof, it is enough to show
\[
u \atimes v \subseteq s \mand u' \atimes v' \subseteq t \linimp
(u \acap u') \atimes (v \acap v') \subseteq s \acap t.
\]
This is also true in affine logic.
\end{proof}
\begin{remark}
To prove \textbf{I3} for $\ainte$ by a similar argument to that for $\minte$, it is necessary to show $\inte u \atimes \inte v \subseteq \ainte s$ for $u:\P{X},v:\P{Y}$ such that $u \atimes v \subseteq s$. However, we can only prove $\inte u \mtimes \inte v \subseteq \ainte s$.
\end{remark}
\section{Compactness}\label{sec:compactness}
In this section, we define an affine version of compactness and prove its basic properties.
We use the notions of filter and cluster points to define compactness. Compared with the definition via open coverings, this definition has the advantage that we do not have to quantify index types.
Note that, our definitions apply not only to topological spaces but to general operators.

First, we define the notion of filters.


\def\PFil#1{\mathbf{Fil}\left(#1\right)}

\begin{definition}\label{def:filter}
\def\F{\mathcal{F}}
A subset of the power set $\F: \P{\P{X}}$ is called a \textit{filter on $X$} if it is monotonic, non-empty, and closed under multiplicative intersection with $\ofcourse$-modality.
\[
\PFil{\F}:=\ofcourse{\left(\begin{array}{l}\left(\forall s,t : \P{X}, \left(s\in \F \linimp s \subseteq t\linimp t\in \F\right)\right) \mand X \in \F \\ 
 \mand \forall s,t : \P{X}, \left(s\in \F \mand t \in \F \linimp s \mcap t\in \F\right)
 \end{array}\right)}.
\]
\end{definition}

This definition is justified by the following lemma.

\def\Cov{\mathbf{Cov}}

\begin{lemma}\label{lemma:isFilter}
Let $A$ be any type.
For any indexed family $U:A \to \P{X}$ of subsets,
\begin{equation}\label{eq:isFilter}
\PFil{\Set{s:\P{X}}{\exists n : \N, \exists F :[n]\to A, \compl{s} \subseteq \bigmcup_{i:[n]}U_{F(i)}}}.
\end{equation}
\end{lemma}
\begin{proof}
    Clear.
\end{proof}

The definition of compactness is as follows.

\def\MCptd{\mathbf{Cptd}}
\def\MCpt{\mathbf{Cpt}}

\begin{definition}
\def\F{\mathcal{F}}
Let $X$ be a type equipped with an operator $\inte: \P{X} \to \P{X}$.
For a filter $\F$ in $X$, a point $x : X$ is called a \textit{cluster point} of $\F$ if it is contained in the closure of any element of $\F$.
A subset $s:\P{X}$ is \textit{compact} if, for any filter $\F$ such that $\compl{s}\notin \F$, there exists a point in $s$ which is a cluster point with $\ofcourse$-modality:
\begin{align*}
\mathbf{Clst}(\F, x) &:= \forall t : \P{X}, \left(t \in \F \linimp x \in \clo t\right),\\
\MCpt(s) &:= \forall \F : \P{\P{X}}, (\PFil{\F} \linimp \compl{s} \notin \F
\linimp \exists x : X, \left(x \in s \mand \ofcourse{\mathbf{Clst}(\F, x)})\right).
\end{align*}
\end{definition}

    

\begin{proposition}
Compactness is invariant under affine equivalence:
\[
\MCpt(s) \linimp s\subseteq t \mand t \subseteq s \linimp \MCpt(t).
\]
\end{proposition}
\begin{proof}
Obvious from the definition.
\end{proof}
\begin{proposition}\label{prop:compact has finite covering}
For an indexed family $U: A \to \P{X}$ of subsets,
\[
\MCpt(s)\linimp s \subseteq \subsetwn{\left(\bigcup_{\alpha : A} \inte U_{\alpha}\right)} \linimp \exists n : \N, \exists F :[n]\to A, s\subseteq \bigmcup_{i:[n]}U_{F(i)}.
\]
\end{proposition}
\begin{proof}
Let $\mathcal{F}$ be as in (\ref{eq:isFilter}). To show $\compl{s} \in \mathcal{F}$, by Lemma~\ref{lemma:isFilter} and the definition of compactness, it is enough to deduce $\forall x:X, x \in s \linimp \whynot{\exists t : \P{X}, t \in \mathcal{F} \mand x \notin \clo t}$,
which follows from the assumption $s \subseteq \subsetwn{\left(\bigcup_{\alpha : A} \inte U_{\alpha}\right)}$ because $\compl{U_{\alpha}}\in \mathcal{F}$ for any $\alpha:A$.
\end{proof}

Clearly, when all $U_{\alpha}$ are open, the above proposition corresponds to the classical fact that any open covering has a finite subcovering.
Note that, a similar result can be shown even if the condition for the cluster point does not have exponential conjunction.
This modality is required for the following propositions, which are affine analogs of the classical proposition that a closed subset of a compact set is also compact.

\begin{proposition}\label{prop:closed set in compact is compact}
The additive intersection of a compact subset $s$ and a closed decidable subset $c$ is also compact:
\[
\MCpt(s) \linimp \clo c \subseteq c \linimp X \subseteq c \acup \compl{c}\linimp \MCpt(s \acap c).
\]
\end{proposition}
\begin{proof}
Let $\mathcal{F}$ be a filter on $X$.
Let
\[
\mathcal{F}' := \Set{t:\P{X}}{t\acup \compl{c} \in \mathcal{F}}.
\]
Then, $\mathcal{F}'$ is also a filter.
Indeed, monotonicity and non-emptiness are clear.
If $t_0, t_1 \in \mathcal{F}'$, since
\[
(t_0 \acup \compl{c}) \mcap (t_1 \acup \compl{c}) \subseteq(t_0\mcap t_1)\acup \compl{c},
\]
we have $t_0\mcap t_1 \in \mathcal{F}'$.

The condition $\compl{\left(s\acap c\right)} \notin \mathcal{F}$ means $\compl{s}\notin\mathcal{F}'$. Since $s$ is compact, we can take $x\in s$ which is a cluster point of $\mathcal{F}'$ with $\ofcourse$-modality. Since $c$ is decidable, $c \in \mathcal{F}'$. Thus, $x \in \clo c$. Since $c$ is closed, $x \in c$. Furthermore, $x$ is also a cluster point of $\mathcal{F}$ with $\ofcourse$-modality since $\mathcal{F}\subseteq \mathcal{F}'$.
\end{proof}
Note that, the fact that $x$ is a cluster point of $\mathcal{F}'$ is used twice in the above proof.

A similar proposition holds for the multiplicative intersection under a different assumption.
\begin{proposition}\label{prop:closed set in compact is compact, multiplicative}
The multiplicative intersection of a compact subset $s$ and a closed subset $c$, such that $\ofcourse{c\subseteq c\mcap c}$, is also compact:
\[
\MCpt(s) \linimp \clo c \subseteq c \linimp \ofcourse{\left(c \subseteq c\mcap c\right)}\linimp \MCpt(s \mcap c).
\]
\end{proposition}
\begin{proof}
Let $\mathcal{F}$ be a filter on $X$.
Let
\[
\mathcal{F}' := \Set{t:\P{X}}{t\mcup \compl{c} \in \mathcal{F}}.
\]
Then, $\mathcal{F}'$ is also a filter.
Indeed, monotonicity and non-emptiness are clear.
The following holds with $\ofcourse$-modality for $t_0, t_1 \in \mathcal{F}'$ by the assumption that $\ofcourse{c\subseteq c \mcap c}$:
\[
(t_0 \mcup \compl{c}) \mcap (t_1 \mcup \compl{c}) \subseteq(t_0\mcap t_1)\mcup \compl{c}\mcup \compl{c} \subseteq (t_0\mcap t_1)\mcup \compl{c}.
\]
Thus, we have $t_0\mcap t_1 \in \mathcal{F}'$.

The condition $\compl{\left(s\mcap c\right)} \notin \mathcal{F}$ means $\compl{s}\notin\mathcal{F}'$. Since $s$ is compact, we can take $x\in s$ which is a cluster point of $\mathcal{F}'$ with $\ofcourse$-modality. Clearly, $c \in \mathcal{F}'$. Thus, $x \in \clo c$. Since $c$ is closed, $x \in c$. Furthermore, $x$ is also a cluster point of $\mathcal{F}$ with $\ofcourse$-modality since $\mathcal{F}\subseteq \mathcal{F}'$.
\end{proof}





Next, we consider continuous images of compact sets.

\begin{definition}
Let $f:X\to Y$ be a function.
For a subset $t:\P{Y}$, the \textit{inverse image} $f^{-1}(t)$ of $t$ under $f$ is defined as the functional composition $t\circ f$.
For a subset $s:\P{X}$, the \textit{image} $f(s)$ of $s$ under $f$ is impredicatively defined as
\[
f(s):=\Set{y:Y}{\forall t:\P{Y}, \left(s \subseteq f^{-1}(t) \linimp y \in t\right)}.
\]
\end{definition}

\begin{remark}
While a more natural and predicative way to define the image is
\[f(s)':=\Set{y:Y}{\exists x:X, \left(x\in s \mand f(x) = y\right)},\]
we avoid using equality. As long as the equality is reflexive, $f(s)\subseteq f(s)'$ holds. However, to show the converse, it is necessary that, for all $t:\P{X}$, $f(x)\in t\mand f(x)=y \linimp y \in t$.
\end{remark}

\begin{definition}
Let $X$ and $Y$ be sets with operators.
A function $f:X\to Y$ is \textit{continuous} if
\[
\mathbf{Conti}(f) := \left(\forall t : \P{Y}, f^{-1}(\inte t) \subseteq \inte f^{-1}(t) \right).
\]
\end{definition}
Note that this is equivalent to $\forall t : \P{Y}, \clo f^{-1}(t) \subseteq f^{-1}(\clo t)$.

\begin{theorem}\label{theorem:image of compact is compact}
For sets $X$ and $Y$ with operators, a function $f:X\to Y$, and a subset $s:\P{X}$, if $f$ is continuous with $\ofcourse$-modality and $s$ is compact, then the image of $s$ under $f$ is compact:
\[
\ofcourse{\mathbf{Conti}(f)}\linimp\MCpt(s) \linimp \MCpt(f(s)).
\]
\end{theorem}
\begin{proof}
\def\F{\mathcal{F}}
For $\F:\P{\P{Y}}$ which is a filter, let
\[\F':=\Set{t:\P{X}}{\exists u:\P{Y}, u\in \F \mand f^{-1}(u)\subseteq t}.\]

First, we prove that $\F'$ is a filter on $X$.
Monotonicity is clear.
We can show $X \in \F'$ by taking $Y$ as $u$.
Since $f^{-1}$ commutes with the multiplicative intersection, the closedness under intersections for $\F'$ is deduced from the same condition for $\F$. 

Second, we prove $\compl{f(s)}\notin \F \linimp \compl{s} \notin \F'$ with contraposition. If there exists $u:\P{Y}$ such that $u\in\F\mand f^{-1}(u)\subseteq \compl{s}$, since $f^{-1}(u)\subseteq \compl{s}$ is equivalent to $s\subseteq f^{-1}(\compl{u})$, we obtain $f(s)\subseteq \compl{u}$, which is equivalent to $u \subseteq \compl{f(s)}$.
Thus, $\compl{f(s)}\in \F$.

Therefore, when $\compl{f(s)}\notin \mathcal{F}$, since $s$ is compact, we can take $x:X$ in $s$ which is a cluster point of $\F'$ with $\ofcourse$-modality. It is enough to show $f(x)$ is a cluster point of $\F$ with $\ofcourse$-modality. Let $t\in \F$. Since $f^{-1}(t)\in \F'$ and $x$ is a cluster point, $x \in \clo f^{-1}(t)$. Since $f$ is continuous, $f(x)\in\clo t$.
\end{proof}

\begin{remark}
In Theorem~\ref{theorem:image of compact is compact}, while continuity on $s$ is enough for the condition of $f$ classically, we require continuity on the whole set for $f$ because we already use $x\in s$ for $f(x)\in f(s)$.
\end{remark}

The following proposition corresponds to the finite Tychonoff theorem in classical mathematics.

\begin{proposition}\label{prop:compactness_product}
Let $X,Y$ be spaces with operators, and we assume that the operator on $X$ satisfies \textbf{I4} and \textbf{I5} in Definition~\ref{def:operator}.
For a compact subset $s_X$ of $X$ and a compact subset $s_Y$ of $Y$, the additive product $s_X\atimes s_Y$ is compact with respect to $\ainte$:
\[
\MCpt(s_X)\mand\MCpt(s_Y)\linimp\MCpt\left(s_X\atimes s_Y\right).
\]
\end{proposition}
\begin{proof}
In this proof, for $u:\P{X}$ and $v:\P{Y}$, we write
\[
u\catimes v := \compl{\left(\compl{u}\atimes\compl{v}\right)}
\equiv \Set{(x,y):X\times Y}{x\in u \aor y \in v}.
\]

For a filter $\mathcal{F}$ on $X\times Y$, let
\[
\mathcal{F}_X := \Set{t : \P{X}}{t \catimes \compl{s_Y} \in \mathcal{F}}.
\]
We show that $\mathcal{F}_X$ is a filter on $X$.
It is clear for monotonicity and non-emptiness.
If $t_0, t_1 \in \mathcal{F}_X$, since
\begin{equation}\label{eq:tststts}
\left(t_0 \catimes \compl{s_Y}\right) \mcap \left(t_1 \catimes \compl{s_Y}\right)
\subseteq
(t_0\mcap t_1)\catimes \compl{s_Y},
\end{equation}
we have $t_0\mcap t_1 \in \mathcal{F}_X$.
Clearly, $\compl{\left(s_X\atimes s_Y\right)}\notin \mathcal{F}$ is equivalent with $\compl{s_X} \notin \mathcal{F}_X$.

Thus, if $\compl{\left(s_X\atimes s_Y\right)}\notin\mathcal{F}$ and $s_X$ is compact, we can take $x\in s_X$ which is a cluster point of $\mathcal{F}_X$ with $\ofcourse$-modality.

Let
\[
\mathcal{F}_Y := \Set{t : \P{Y}}{\exists u : \P{X},\, u \catimes t \in \mathcal{F} \mand x \notin \clo u}.
\]
The fact that $x$ is a cluster point of $\mathcal{F}_X$ means $\compl{s_Y} \notin \mathcal{F}_Y$.

We show that $\mathcal{F}_Y$ is a filter on $Y$.
The monotonicity is clear.
Non-emptiness can be shown by taking $\emptyset$ as $u$, because $\clo \emptyset \subseteq \emptyset$ in $X$ from \textbf{I4}.
For $t_0, t_1 \in \mathcal{F}_Y$, we take $u_0, u_1 : \P{X}$ satisfying the condition, respectively.
Since
\begin{equation}\label{eq:ututuutt}
\left(u_0 \catimes t_0\right) \mcap \left(u_1 \catimes t_1\right)
\subseteq
(u_0\acup u_1)\catimes (t_0\mcap t_1),
\end{equation}
it is enough to show \[x \notin \clo u_0 \mand x \notin\clo u_1 \linimp x \notin \clo (u_0\acup u_1).\]
This is a direct consequence of \textbf{I5}.

Thus, if $s_Y$ is also compact, we can take $y \in s_Y$ which is a cluster point of $\mathcal{F}_Y$ with $\ofcourse$-modality.
In other words, for any $u : \P{X}$ and $v : \P{Y}$,
\[
u \catimes v \in \mathcal{F} \linimp
\left(x \in \clo u \mor y \in \clo v\right).
\]
This means that the point $(x,y) \in s_X \atimes s_Y$ is a cluster point of $\mathcal{F}$ with respect to $\ainte$ with $\ofcourse$-modality. Indeed, for any $w\in\mathcal{F}$, if $u:\P{X}$ and $v:\P{Y}$ satisfy $u\atimes v \subseteq \compl{w}$, then $\compl{u}\catimes\compl{v}\in\mathcal{F}$, so we have $x\notin \inte u \mor y \notin \inte v$, which means $(x,y)\notin \inte u \mtimes \inte v$.
Thus, $(x,y) \notin \ainte \compl{w}$.

\end{proof}

\begin{remark}
While we can give a similar argument for $s_X\mtimes s_Y$ and $\minte$, additional assumptions are necessary.
First, for the inclusion corresponding to (\ref{eq:tststts}), $s_Y$ should satisfy $\ofcourse{\left(s_Y\subseteq s_Y\mcap s_Y\right)}$.
Second, for the inclusion corresponding to (\ref{eq:ututuutt}), we need to use $u_0\mcup u_1$ instead of $u_0\acup u_1$.
Therefore, instead of \textbf{I5}, the stronger assumption that $\inte u\mcap \inte v \subseteq \inte (u\mcap v)$ is necessary.
\end{remark}
\begin{remark}
The monotonicity assumption for a filter is used more than once in several proofs.
Other conditions are also used twice in the proof in Proposition~\ref{prop:compactness_product}.
Thus, these conditions must appear with the exponential conjunction in the definition of filters (Definition~\ref{def:filter}).
\end{remark}
\section{Topology on Cuts}\label{sec:cuts}

\def\Cuts{\mathcal{C}}
\def\EP{\mathcal{I}}

In this section, we define the topology on cuts and prove the propositions corresponding to the extreme-value theorem and the Heine-Borel theorem in classical mathematics.

First, we review the notion of cuts in \cite{shulman2022affine}.
For a subset $s:\P{\Q}$ of rational numbers,
\begin{itemize}
\item $s$ is a \textit{lower set} if $\forall a,b:\Q, \left(a<b \mand b\in s\linimp a \in s\right)$,
\item $s$ is \textit{upwards-open} if $\forall a:\Q, \left(a\in s \linimp \exists b: \Q, \left(a<b\mand b \in s\right)\right)$,
\item $s$ is \textit{upwards-closed} if $\forall b:\Q, \left(\left(\forall a:\Q, \left(a < b \linimp a\in s\right)\right)\linimp b \in s\right)$,
\item $s$ is an \textit{upper set} if $\forall a,b:\Q, \left(a<b \mand a\in s\linimp b \in s\right)$,
\item $s$ is \textit{downwards-open} if $\forall b:\Q, \left(b\in s \linimp \exists a: \Q, \left(a<b\mand a \in s\right)\right)$,
\item $s$ is \textit{downwards-closed} if $\forall a:\Q, \left(\left(\forall b:\Q, \left(a < b \linimp b\in s\right)\right)\linimp a \in s\right)$.
\end{itemize}
The four types---an upwards-open lower subset $L$, an upwards-closed lower subset $\overline{L}$, a downwards-open upper subset $U$, and a downwards-closed upper subset $\overline{U}$---have a one-to-one correspondence with each other (Theorem 9.4 in \cite{shulman2022affine}). Namely,
\begin{align*}
L &\mapsto \overline{L} := \Set{b:\Q}{\forall a: \Q, \left(a<b \linimp a\in L\right)}, \\ \overline{L} &\mapsto L:=\Set{a:\Q}{\exists b: \Q, \left(a<b \mand b\in \overline{L}\right)}, \\
U &\mapsto \overline{U} := \Set{a:\Q}{\forall b: \Q, \left(a<b \linimp b\in U\right)}, \\ \overline{U} &\mapsto U:=\Set{b:\Q}{\exists a: \Q, \left(a<b \mand a\in \overline{U}\right)}, \\
L &\leftrightarrow \overline{U}:=\compl{L}, \\U &\leftrightarrow \overline{L}:=\compl{U}.
\end{align*}
Let $\mathcal{C}$ be one of these four types and we call its elements \textit{cuts}.
A cut $x:\mathcal{C}$ can be represented by any of the four subsets $x_L$, $x_{\overline{L}}$, $x_U$, and $x_{\overline{U}}$.

For $x,y:\mathcal{C}$, we define the order as
\[
x\leq y := (x_L\subseteq y_L)\equiv(x_{\overline{L}}\subseteq y_{\overline{L}})\equiv(y_U\subseteq x_U)\equiv(y_{\overline{U}}\subseteq x_{\overline{U}})
\]
and $x<y:=\not{(y\leq x)}$.
A rational number $q:\Q$ can be considered as a cut by $q_L:=\Set{r:\Q}{r<q}$.
For $q : \Q$ and $x : \Cuts$, $q<x$, $q\leq x$, $x < q$, and $x \leq q$ are equivalent to $q\in x_L$, $q \in x_{\overline{L}}$, $q\in x_U$, and $q\in x_{\overline{U}}$, respectively.

Next, we define the topology on $\Cuts$ using intervals.
We stipulate that $-\infty < x$, $-\infty \leq x$, $x < \infty$, and $x \leq \infty$ are always true (have truth value $\top$).
Let $\EP := \left(\Q \cup \{-\infty\}\right) \times \left(\Q \cup \{\infty\}\right)$ be the type of pairs of interval endpoints.
Here, the union of types is additive. In other words, $\forall(a,b):\EP, P(a,b)$ means
\[\left(\forall a,b:\Q, P(a,b)\right)\aand \left(\forall a, P(a,\infty)\right) \aand \left(\forall b:\Q, P(-\infty, b)\right) \aand P(-\infty,\infty)\]
and $\exists(a,b):\EP, P(a,b)$ means
\[\left(\exists a,b:\Q, P(a,b)\right)\aor \left(\exists a, P(a,\infty)\right) \aor \left(\exists b:\Q, P(-\infty, b)\right) \aor P(-\infty,\infty).\]
We define, for $(a,b) : \EP$,
\[[a,b]:=\Set{x:\mathcal{C}}{a\leq x \mand x \leq b},\]
\[(a,b):=\Set{x:\mathcal{C}}{a< x \mand x < b},\]
\[[a,b):=\Set{x:\mathcal{C}}{a \leq x \mand x < b},\]
\[(a,b]:=\Set{x:\mathcal{C}}{a< x \mand x \leq b}.\]



\begin{proposition}
The family of open intervals $(-,-):\EP\to \Cuts$ is a basis.
Thus, it defines a topology on $\Cuts$.
\end{proposition}
\begin{proof}
The first condition is clear because $(-\infty,\infty)$ contains all cuts.

We show the second condition. We assume $x \in (q_0, r_0) \mand x \in (q_1,r_1)$.
Let $q_2 := \max(q_0,q_1)$ and $r_2:=\min(r_0, r_1)$.
Then, $x\in(q_2, r_2)$, $(q_2,r_2)\subseteq(q_0,r_0)$, and $(q_2,r_2)\subseteq(q_1,r_1)$ hold.
\end{proof}

The following theorem is the affine analogue of the extreme value theorem in classical analysis.
Because two cuts are not always comparable, we additionally assume the existence of an upper bound for arbitrary finite families.

\begin{theorem}[Extreme Value Theorem]\label{theorem:compact set has max}
Let $X$ be a topological space.
We assume that $s:\P{X}$ is a compact subset and $f:X\to \Cuts$ is a continuous function such that for any finite family $(x_i)_{i:[n]}$ in $s$, there exists $x'$ in $s$ such that $f(x_i)\leq f(x')$ for any $i$. Then, $f$ attains its maximum value on $s$:
\begin{align*}
&\MCpt(s)\mand \left(\forall n : \N, x_{\_}: [n] \to X, \left(\bigmand_{i:[n]} x_i\in s \linimp \exists x' : X, \left(x'\in s \mand\bigmand_{i:[n]} f(x_i) \leq f(x')\right)\right)\right)
\\ &\linimp \exists x':X, \left(x'\in s \mand \ofcourse{\forall x:X, \left(x\in s \linimp \mathbf{Conti}(f) \linimp f(x)\leq f(x')\right)}\right).
\end{align*}
\end{theorem}
\begin{proof}
\def\F{\mathcal{F}}
Let
\[\F:=\Set{t:\P{X}}{\exists n: \N, x_{\_} : [n]\to X, \left(\bigmand_{i:[n]} x_i\in s  \mand \bigmcap_{i:[n]} f^{-1}\left([f(x_i),\infty)\right)\subseteq t\right)}.\]

Obviously, $\F$ is a filter. Moreover, the assumption that any finite subset of $s$ has an upper bound in $s$ means that $\compl{s} \notin \F$. Therefore, if $s$ is compact, we can take $x'\in s$ which is a cluster point of $\F$ with $\ofcourse$-modality. For any $x \in s$, since $f^{-1}([f(x), \infty)) \in \F$, $x' \in \clo f^{-1}([f(x),\infty))$.
If $f$ is continuous, $f(x') \in \clo [f(x),\infty)$,
which means that, for any $(q,r)\in \EP$, if $(q,r) \subseteq (-\infty,f(x))$, then $f(x') \notin (q,r)$. Thus, for any $r : \Q$, $r \leq f(x) \linimp r \leq f(x')$, which means $f(x) \leq f(x')$.
\end{proof}
Note that, in the above theorem, the assumption that $f$ is continuous is moved to the position where it is used in the proof.
If $f$ is continuous with $\ofcourse$-modality, we can conclude that $\ofcourse{\forall x:X, \left(x\in s \linimp f(x)\leq f(x')\right)}$.
\begin{remark}
To reduce the upper bound condition to the case of two elements, we have to strengthen the expression because it is used arbitrarily many times to show the condition for arbitrary $n$. More precisely, the condition
\[
\ofcourse{\forall x,y:X, \left( x\in s \mand y \in s \linimp \exists z :X, \left(z\in s \mand \ofcourse{\left(f(x)\leq f(z)\right)}\mand f(y)\leq f(z)\right)\right)}
\]
implies the upper bound condition in the above theorem by induction for $n$.
\end{remark}

We can also prove the analogue of the Heine-Borel theorem, whose proof is essentially dependent on the assumption of propositional impredicativity. For the proof, we first state a lemma.

\begin{lemma}\label{lemma:interval_union}

For $(a,c)\in \EP$ and $b : \Q$,
\[[a,c) \subseteq [a,b]\mcup (b,c).\]
\end{lemma}
\begin{proof}
It follows from the definition of intervals and affine logic:
\[a\leq x\mand \not{(a\leq x \mand x \leq b)} \mand x < c\linimp b < x \mand x < c.\]
\end{proof}

\begin{theorem}[Heine-Borel Theorem for Cuts]\label{theorem:heine-borel}
For any $a,b\in \Q$, $[a,b] :\P{\Cuts}$ is compact: \[\forall a,b\in \Q,\MCpt([a,b]).\]
\end{theorem}
\begin{proof}
\def\F{\mathcal{F}}
    Let $\F : \P{\P{\Cuts}}$ be a filter such that $\compl{[a,b]}\notin \F$. We can define an element $x$ of $\Cuts$ as
    \begin{equation}\label{eq:defining cluster point}
    x_{\overline{L}} := \Set{q:\Q}{\forall r : \Q, r<q\linimp \compl{[a,r]}\in \F},
    \end{equation}
    because the right-hand side is obviously an upwards-closed lower set.
    The inequality $a\leq x$ follows from the fact that, for any $r: \Q$ such that $r<a$, we have $[a,r]\subseteq \emptyset$.
    Since $\compl{[a,b]}\notin \F$, we have $x\leq b$. Thus $x\in [a,b]$.

    It is enough to show that $x$ is a cluster point of $\F$ with $\ofcourse$-modality. Let $t\in \F$. If $x \notin \clo t$, there exists $(q,r) : \EP$ such that $q<x<r \mand (q,r) \subseteq \compl{t}$.
    Since $t \in \F$, we have $\compl{(q,r)} \in \F$.

    If $q$ is $-\infty$, since $[a, r) \subseteq (q, r)$, clearly $\compl{[a,r)}\in \F$.
    Otherwise, by definition of $x$ and $q<x$, we have $\compl{[a,q]}\in \F$. 
    Thus, $\compl{[a,q]}\mcap\compl{(q,r)} \in \F$.
    On the other hand, by Lemma~\ref{lemma:interval_union}, $\compl{[a,q]}\mcap\compl{(q,r)}\subseteq\compl{[a,r)}$.
    Therefore, $\compl{[a,r)}\in \F$.
    
    For any $r'<r$, since $[a,r']\subseteq [a,r)$, we have $\compl{[a,r']}\in \F$.
    Thus, $r\leq x$, which contradicts $x<r$.
\end{proof}




\section{Antithesis Translation of the Heine-Borel Theorem}\label{sec:translation}

In this section, we consider the antithesis translation of the Heine-Borel theorem in the previous section.

\textit{Antithesis translation}~(\cite{shulman2022affine}) is a transformation from a proposition or a predicate $P$ in affine logic to a pair $(P^+, P^-)$ of the intuitionistic ones which satisfies $P^+\to \lnot P^-$.
It can be defined recursively as follows:
\begin{align*}
(P\aand Q)^+&:=P^+ \land Q^+, &  (P\aand Q)^-&:= P^-\lor Q^-, \\
(P\aor Q)^+&:=P^+ \lor Q^+, &  (P\aor Q)^-&:= P^-\land Q^-, \\
(\not{P})^+ &:= P^-, & (\not{P})^- &:= P^+, \\
(P\mand Q)^+&:=P^+ \land Q^+, &  (P\mand Q)^-&:= (P^+\to Q^-) \land (Q^+\to P^-), \\
(P\mor Q)^+&:=(P^- \to Q^+)\land(Q^- \to P^+), &  (P\mor Q)^-&:= P^-\land Q^-, \\
\top^+ &:= \top, & \top^- &:= \bot, \\
\bot^+ &:= \bot, & \bot^- &:= \top, \\
\left(\ofcourse{P}\right)^+ &:= P^+, & \left(\ofcourse{P}\right)^- &:= \lnot P^+, \\
\left(\whynot{P}\right)^+ &:= \lnot P^-, & \left(\whynot{P}\right)^- &:= P^-, \\
(\exists x, P)^+ &:= \exists x, P^+, & (\exists x, P)^- &:= \forall x, P^-,\\
(\forall x, P)^+ &:= \forall x, P^+, & (\forall x, P)^- &:= \exists x, P^-.
\end{align*}
If $P$ is provable in affine logic, $P^+$ is provable in intuitionistic logic.

Since the complete translation of the Heine-Borel theorem is complex, we consider the following simplified corollary.

\begin{corollary}\label{cor:corollary}
For any two rational numbers $a, b : \Q$ and two indexed families $q, r : A \to \Q$ of rational numbers,
if $[a,b] \subseteq \subsetwn{\left(\bigcup_{\alpha:A} (q_{\alpha}, r_{\alpha})\right)}$, then there exist a natural number $n : \N$ and $n$ indices $F : [n] \to A$ such that $[a,b] \subseteq \bigmcup_{i:[n]} (q_{F(i)}, r_{F(i)})$.
\end{corollary}
\begin{proof}
It follows from Theorem~\ref{theorem:heine-borel}, Proposition~\ref{prop:compact has finite covering}, and the fact that any open interval is open.
\end{proof}

To make the translation explicit, we check the translation of the relevant notions one by one.
From here in this section, all definitions and theorems will be considered within intuitionistic logic.

According to Theorem 9.6 in \cite{shulman2022affine},
the antithesis translation of the cuts is equivalent to the notion called \textit{(rational) cuts} in \cite{richman1998generalized}, which is also called the \textit{interval domain}:

\begin{definition}
A set $L$ of rational numbers is called a \textit{lowercut} if
\begin{itemize}
\item $L$ is a \textit{lower set}: $\forall a,b : \Q, \left(a < b \land b \in L \to a \in L\right)$, and
\item $L$ is \textit{upwards-open}: $\forall a : \Q, \left(a \in L \to \exists b : \Q, \left(a < b \land b \in L\right)\right)$.
\end{itemize}
A set $U$ of rational numbers is called an \textit{uppercut} if
\begin{itemize}
\item $U$ is an \textit{upper set}: $\forall a,b : \Q, \left(a < b \land a \in U \to b \in U\right)$, and
\item $U$ is \textit{downwards-open}: $\forall b : \Q, \left(b \in U \to \exists a : \Q, \left(a < b \land a \in U\right)\right)$.
\end{itemize}
A pair $x=(L, U)$ of a lowercut $L$ and an uppercut $U$ is called a \textit{cut} if \[\forall a,b : \Q, \left(a \in L \land b \in U \to a < b\right).\]
\end{definition}
Lowercuts and uppercuts are also called \textit{extended lower reals} and \textit{extended upper reals}, respectively~(\cite{levsnik2021synthetic}).

\def\InCuts{\mathfrak{I}}
Let $\mathfrak{L}, \mathfrak{U}$, and $\InCuts$ be the type of lowercuts, uppercuts, and (intuitionistic) cuts, respectively. For $L:\mathfrak{L}$ and $U:\mathfrak{U}$, we write
\begin{align*}
\overline{L} &:= \Set{q:\Q}{\forall r : \Q, \left(r<q \to r\in L\right)},\\
\overline{U} &:= \Set{q:\Q}{\forall r : \Q, \left(q<r \to r\in U\right)}.
\end{align*}

For a cut $x=(L, U) : \InCuts$ and a rational number $q : \Q$, we write $q \in L$, $q \in \overline{L}$, $q \in U$, and $q \in \overline{U}$ as $q < x$, $q\leq x$, $x < q$, and $x \leq q$, respectively, which correspond to the same expressions as in affine logic.

A subset $s : \P{X}$ in affine logic is translated to the following:
\begin{definition}
A \textit{complemented subset} of a type $X$ is a pair of subsets $U=(U^{+}, U^{-})$ of $X$ such that $x\in U^{+} \to x \notin U^{-}$ for any $x:X$.
\end{definition}
Note that, because we do not consider equality relations, the above definition is simpler than that in \cite{shulman2022affine}.

The inclusion relation of subsets is translated to the following:
\begin{definition}
For two complemented subsets $U=(U^{+}, U^{-}), V=(V^{+}, V^{-})$ in $X$, $U$ is \textit{included} in $V$ if $U^{+} \subseteq V^{+}$ and $V^{-} \subseteq U^{-}$, which is denoted as $U\subseteq V$.
\end{definition}

The multiplicative union of subsets is translated as follows: 
\begin{definition}
The \textit{multiplicative union} of two complemented subsets $U=(U^{+}, U^{-})$, $V=(V^{+}, V^{-})$ in $X$ is defined as follows:
\[
U \mcup V := \left(\Set{x:X}{(x\in U^{-} \to x \in V^{+}) \land (x \in V^{-} \to x \in U^{+})}, U^{-} \cap V^{-}\right).
\]
The multiplicative union of a finite family of complemented subsets is inductively defined.
\end{definition}

The intervals are translated to the following complemented subsets in cuts:
\begin{align*}
[a,b]_{\mathrm{cut}} &:= \left(\Set{x:\InCuts}{a \leq x \land x \leq b},\Set{x:\InCuts}{(a \leq x \to b < x) \land (x \leq b \to x < a)}\right),\\
(a,b)_{\mathrm{cut}} &:= \left(\Set{x:\InCuts}{a < x \land x < b},\Set{x:\InCuts}{(a < x \to b \leq x) \land (x < b \to x \leq a)}\right),\\
[a,b)_{\mathrm{cut}} &:= \left(\Set{x:\InCuts}{a \leq x \land x < b},\Set{x:\InCuts}{(a \leq x \to b \leq x) \land (x < b \to x < a)}\right),\\
(a,b]_{\mathrm{cut}} &:= \left(\Set{x:\InCuts}{a < x \land x \leq b},\Set{x:\InCuts}{(a < x \to b < x) \land (x \leq b \to x \leq a)}\right).
\end{align*}
Note that it can be directly shown that $b<c\to [a,b]_{\mathrm{cut}} \subseteq [a,c)_{\mathrm{cut}}$, etc.

We also have to consider the translation of $A \linimp \whynot{B}$, where $A$ and $B$ are propositions.
Let $(A^+, A^-)$ and $(B^+, B^-)$ be the translation of $A$ and $B$, respectively.
The positive part of the translation of $A \linimp \whynot{B}$ is $A^+\to \lnot B^- \land B^-\to A^-$.
However, $A^+ \to \lnot{B^-}$ follows from $B^-\to A^-$ since $A^+ \to \lnot{A^-}$.
Thus, this positive part is equivalent to $B^-\to A^-$.

For simplicity, we state only half of the whole translation of the Corollary~\ref{cor:corollary}:
it has the form of $A\linimp B$ and the positive part of the translation of $A\linimp B$ is $A^+\to B^+ \land B^-\to A^-$,
where $(A^+, A^-)$ and $(B^+, B^-)$ are the translations of $A$ and $B$, respectively.
We only state the $A^+\to B^+$ part of the translation of the above corollary.

Finally, the translated statement is the following:

\begin{quote}
Let $a, b$ be two rational numbers and let $(q_\alpha)_{\alpha:A}, (r_\alpha)_{\alpha:A}$ be two indexed families of rational numbers.
We assume that $\bigcap_{\alpha:A} (q_{\alpha}, r_{\alpha})_{\mathrm{cut}}^{-} \subseteq [a,b]_{\mathrm{cut}}^{-}$.
Then there exist a natural number $n : \N$ and $n$ indices $F : [n] \to A$ such that $[a,b]_{\mathrm{cut}} \subseteq \bigmcup_{i:[n]} (q_{F(i)}, r_{F(i)})_{\mathrm{cut}}$.
\end{quote}

However, upon examining the proof, we can weaken the assumption.
The condition $\bigcap_{\alpha:A} (q_{\alpha}, r_{\alpha})_{\mathrm{cut}}^{-} \subseteq [a,b]_{\mathrm{cut}}^{-}$ can be separated into the following conditions about lowercuts and uppercuts.
\begin{itemize}
\item[\textbf{CovL}:] For any lowercut $L:\mathfrak{L}$, if $q_{\alpha}\in L\to r_{\alpha}\in \overline{L}$ for all $\alpha : A$,  then $a \in \overline{L} \to b \in L$.
\item[\textbf{CovU}:] For any uppercut $U:\mathfrak{U}$, if $r_{\alpha}\in U\to q_{\alpha}\in \overline{U}$ for all $\alpha : A$,  then $b \in \overline{U} \to a \in U$.
\end{itemize}
Only \textbf{CovL} is actually used in the proof.
Thus, we have the following theorem.

\begin{theorem}\label{theorem:translated}
Let $a, b$ be two rational numbers and let $(q_\alpha)_{\alpha:A}, (r_\alpha)_{\alpha:A}$ be two indexed families of rational numbers,
such that \textbf{CovL} holds.
Then there exist a natural number $n : \N$ and $n$ indices $F : [n] \to A$ such that $[a,b]_{\mathrm{cut}} \subseteq \bigmcup_{i:[n]} (q_{F(i)}, r_{F(i)})_{\mathrm{cut}}$.
\end{theorem}

\begin{proof}
Let
\begin{equation}\label{eq:HNLower}
L:=\Set{c:\Q}{\exists d:\Q,c<d\land\exists n:\N, F:[n]\to A, [a,d]_{\mathrm{cut}}\subseteq \bigmcup_{i:[n]}(q_{F(i)},r_{F(i)})_{\mathrm{cut}}}.
\end{equation}
Obviously, $L$ is a lowercut.

We show that $q_{\alpha}\in L \to r_{\alpha}\in \overline{L}$.
We take $d$, $n$, $F$ such that $q_{\alpha}<d$ and $[a,d]_{\mathrm{cut}}\subseteq \bigmcup_{i:[n]}(q_{F(i)},r_{F(i)})_{\mathrm{cut}}$.
To show $r_\alpha \in\overline{L}$, we take $c < r_\alpha$ and prove $c \in L$.
Let $e:=(c+r_{\alpha})/2$ and it is enough to show
\[
[a,e]_{\mathrm{cut}}\subseteq \left(\bigmcup_{i:[n]}(q_{F(i)},r_{F(i)})_{\mathrm{cut}}\right)\mcup (q_\alpha,r_\alpha)_{\mathrm{cut}}.
\]
Because $[a,e]_{\mathrm{cut}} \subseteq [a,r_{\alpha})_{\mathrm{cut}}$ and $[a,q_{\alpha}]_{\mathrm{cut}}\subseteq[a,d]_{\mathrm{cut}}\subseteq \bigmcup_{i:[n]}(q_{F(i)},r_{F(i)})_{\mathrm{cut}}$, it is enough to show $[a,r_{\alpha})_{\mathrm{cut}}\subseteq [a,q_{\alpha}]_{\mathrm{cut}}\mcup (q_{\alpha},r_{\alpha})_{\mathrm{cut}}$.
Unfolding the definitions, this inclusion means that, for any $x:\InCuts$,
\begin{align*}
a\leq x \land x < r_{\alpha} \to &((a\leq x\to q_{\alpha}<x)\land(x\leq q_{\alpha} \to x<a)\to q_{\alpha}<x\land x<r_{\alpha})\\ &\land((q_\alpha<x\to r_\alpha\leq x) \land (x<r_\alpha \to x\leq q_{\alpha})\to a \leq x \land x \leq q_{\alpha})
\end{align*}
and
\begin{align*}
&(a\leq x\to q_{\alpha}<x)\land(x\leq q_{\alpha} \to x<a)\land(q_\alpha<x\to r_\alpha\leq x) \land (x<r_\alpha \to x\leq q_{\alpha})\\
&\to (a\leq x\to r_{\alpha} \leq x)\land(x < r_{\alpha} \to x<a).
\end{align*}
These propositions can be verified directly.

Thus, $L$ satisfies the condition in \textbf{CovL}.
Moreover, we have $a\in\overline{L}$ because, for any $c<a$, with $d:=(c+a)/2$ we have $c<d$ and $[a,d]_{\mathrm{cut}}\equiv(\emptyset, \InCuts)$. Therefore, $b\in L$ and the conclusion holds true.
\end{proof}

In particular, the conclusion contains $\bigcap_{i:[n]} (q_{F(i)}, r_{F(i)})_{\mathrm{cut}}^{-} \subseteq [a,b]^{-}_{\mathrm{cut}}$, which contains \textbf{CovL} for the finite covering $(q\circ F, r\circ F)$. This yields the Heine-Borel theorem for lowercuts.

\begin{corollary}[Heine-Borel Theorem for Lowercuts]
If \textbf{CovL} holds, there exists a finite family of indices such that $\textbf{CovL}$ holds even when the indices are restricted to it.
\end{corollary}

Classically, lowercuts correspond to extended real numbers, and \textbf{CovL} is equivalent to $[a,b]\subseteq \bigcup_{\alpha:A} (q_{\alpha}, r_{\alpha})$. Thus, the above corollary is equivalent to the classical Heine-Borel theorem for closed intervals.

Moreover, the conclusion of the theorem also contains \textbf{CovU} for the finite covering $(q\circ F, r\circ F)$. Thus, \textbf{CovL} implies \textbf{CovU}. By repeating the previous arguments with the order of rational numbers reversed, we can prove the Heine-Borel theorem for uppercuts and the fact that \textbf{CovU} also implies \textbf{CovL}.
Therefore, the following holds:
\begin{corollary}
\textbf{CovL} and \textbf{CovU} are equivalent.
\end{corollary}

\begin{remark}
An inhabited ($\exists x:\Q, x\in L$) lowercut $L$ is called a \textit{lower real}, and an inhabited uppercut is called an \textit{upper real}~(\cite{levsnik2021synthetic}).
Clearly, \textbf{CovL} and \textbf{CovU} are relevant only to inhabited ones. Thus, the Heine-Borel theorem for lower reals (or for upper reals) holds true.
\end{remark}

We have formalized,  in Rocq (formerly Coq), the theorem and corollaries together with their proofs, with intuitionistic logic.
The code is available in the repository \url{https://github.com/hziwara/CutsHeineBorel}.

Note that, while the proof of Theorem~\ref{theorem:translated} is intuitionistic, it is not 'constructive,' which means it does not give a way to specify the finite family of indices satisfying the condition of the theorem.
This non-constructivity stems from the use of propositional impredicativity in defining the lowercut.
In other words, it is necessary to interpret the existential quantifier in the conclusion of the theorem as a truncated one,
in the terminology of \cite{hottbook},
because this quantifier is used in (\ref{eq:HNLower}) for defining a term in $\Prop$.

\section{Related Work}\label{sec:related work}
In this section, we discuss the relationship between our results and previous work in constructive mathematics, especially about the Heine-Borel theorem.
In constructive mathematics, it is important to note that due to variations in the definitions of real numbers, topology, and coverings, there exist several non-equivalent formulations of what is called the Heine-Borel theorem.


While our approach is based on point-set, point-free approaches are often taken in constructive topology.
The result that the Heine-Borel does not imply Brouwer's fan theorem in \cite{moerdijk1984heine} and the result that the Heine-Borel holds true in \cite{cederquist1995constructive} (the formulations in these two articles are also different) are within this context and do not treat real numbers as points.
Thus, we cannot simply compare our results with them.

In Russian constructive mathematics, the Heine-Borel theorem has a counterexample (\cite{bridges1987varieties}).
In this framework, real numbers are treated as computable Cauchy sequences, which is different from the Dedekind-style approach we adopt.

In Bishop's constructive mathematics, the use of the Heine-Borel theorem is avoided (\cite{sep-mathematics-constructive}).
In the constructive reverse mathematics, it is proved that the Heine-Borel theorem is equivalent to the fan theorem (\cite{veldman2014principle,diener2018constructive}).
However, this result requires that real numbers can be approximated by rational numbers with arbitrary precision, and one-sided reals and cuts, which we consider, do not satisfy this condition.
In the abstract Stone duality framework, the Heine-Borel theorem is derived from axioms introduced there (\cite{bauer2009dedekind}).

Finally, we discuss the term 'constructive' in the title of this paper.
While some authors use 'constructive' as a synonym for 'intuitionistic', meaning the exclusion of the law of excluded middle, other authors use ‘constructive’ in a sense that implies 'predicative' (\cite{nlab:constructive_mathematics}).
Since we cannot use the term 'intuitionistic,' we use the term 'constructive' in the sense that it allows for impredicativity. However, our result may support the view that constructivity excludes impredicativity, and we may need to coin the term 'affinistic' to describe our standpoint more accurately.

\section{Conclusion}
We developed the theory of compactness in constructive mathematics via affine logic. We proved some basic theorems for compactness, including the Heine-Borel theorem, and verified its translation to intuitionistic logic.
These results imply that this approach to constructive analysis has considerable room for further development.
The definitions we introduce in this paper are tentative, because when defining notions in affine logic, whether to choose additive or multiplicative operations, and whether to apply the exponentials, are delicate issues.

While our definitions are natural in affine logic, their translations to intuitionistic logic are sometimes complicated and difficult to consider without affine logic. Thus, our study confirms that affine logic is an effective approach for constructive topology.
Moreover, the result that the argument in the Heine-Borel theorem essentially depends not on the excluded middle but on propositional impredicativity provides a novel insight into the nature of the continuum.

\bibliography{ref}
\bibliographystyle{plainnat}

\end{document}